\documentclass[11pt]{amsart}
\usepackage{a4wide}
\usepackage{amssymb}
\usepackage{hyperref}
\usepackage{todonotes}
\usepackage[all]{xy}

\newcommand{\Ra}{\Rightarrow}

\newcommand{\U}{\mathcal U}
\newcommand{\w}{\omega}
\newcommand{\e}{\varepsilon}
\newcommand{\supp}{\mathrm{supp}}
\newcommand{\IZ}{\mathbb Z}
\newcommand{\IN}{\mathbb N}
\newcommand{\Aut}{\mathrm{Aut}}
\newcommand{\HM}{\mathsf{HM}}
\newcommand{\TM}{\mathbf{TM_1}}

\newcommand{\IR}{\mathbb R}
\newcommand{\IT}{\mathbb T}
\newcommand{\pr}{\mathrm{pr}}

\newtheorem{theorem}{Theorem}
\newtheorem{proposition}{Proposition}
\newtheorem{lemma}{Lemma}
\newtheorem{corollary}{Corollary}
\newtheorem{problem}{Problem}
\newtheorem{example}{Example}
\theoremstyle{definition}
\newtheorem{definition}{Definition}
\newtheorem{remark}{Remark}

\title[Each topological group embeds into a duoseparable topological group]{Each topological group embeds into\\ a duoseparable topological group}
\author{Taras Banakh, Igor Guran, Alex Ravsky}
\address{T.Banakh: Ivan Franko National University of Lviv (Ukraine), and Jan Kochanowski University in Kielce (Poland)}
\email{t.o.banakh@gmail.com}
\address{I.Guran: Ivan Franko National University of Lviv (Ukraine)}
\email{igor-guran@ukr.net}
\address{O.Ravsky: Pidstryhach Institute for Applied Problems of Mechanics and Mathematics National Academy of Sciences of Ukraine}
\email{alexander.ravsky@uni-wuerzburg.de}

\begin{document}
\begin{abstract} A topological group $X$ is called {\em duoseparable} if there exists a countable set $S\subseteq X$ such that $SUS=X$ for any neighborhood $U\subseteq X$ of the unit. We construct a functor $F$ assigning to each (abelian) topological group $X$ a duoseparable (abelian-by-cyclic) topological group $FX$, containing an isomorphic copy of $X$. In fact, the functor $F$ is defined on the category of unital topologized magmas. Also we prove that each $\sigma$-compact locally compact abelian topological group embeds into a duoseparable abelian-by-countable locally compact topological group.
 \end{abstract}
\maketitle

Let us recall that a topological space is {\em separable} if it contains a dense countable subset. For topological groups, the separability can be characterized as follows.

\begin{theorem}\label{t1} For a topological group $X$ the following conditions are equivalent:
\begin{enumerate}
\item $X$ is separable;
\item $X$ contains a countable subset $S$ such that $X=SU$ for any neighborhood $U\subseteq X$ of the unit;
\item $X$ contains a countable subset $S$ such that $X=US$ for any neighborhood $U\subseteq X$ of the unit;
\item $X$ contains a countable subset $S$ such that $X=USU$ for any neighborhood $U\subseteq X$ of the unit.
\end{enumerate}
\end{theorem}

In this theorem for two subsets $A,B$ of a group $X$ we write $AB=\{ab:a\in A,\;b\in B\}$. Theorem~\ref{t1} (whose proof is left to the reader as an exercise) motivates the following definition, which is the main new concept of this paper.

\begin{definition} A topological group $X$ is defined to be {\em duoseparable} if $X$ contains a countable subset $S$ such that $X=SUS$ for any neighborhood $U\subseteq X$ of the unit.
\end{definition}

It turns out that the duoseparability is not equivalent to the separability and moreover, each topological group $X$ is a subgroup of some duoseparable topological group $FX$, which can be constructed in a functorial way. In fact, the functor $F$ acts in the category of unital topologized magmas, which is much wider than the category of topological groups.

Following Bourbaki, by a {\em magma} we understand a set $X$ endowed with a binary operation\break $\cdot :X\times X\to X$, \; $\cdot:(x,y)\mapsto xy$. 
A magma $X$ is called {\em unital} if $X$ contains a two-sided unit, i.e., an element $1_X\in X$ such that $1_X\cdot x=x\cdot 1_X=x$ for all $x\in X$.  The element $1_X$ is unique and is called {\em the unit} of the unital magma.

A magma is a {\em semigroup} if its binary operation is associative. Unital semigroups are usually called {\em monoids}.

A {\em topologized magma} is a topological space $X$ endowed with a binary operation $\cdot:X\times X\to X$. A topologized magma is called a {\em topological magma} if its binary operation $\cdot:X\times X\to X$ is continuous. 

Unital topologized magmas are objects of the category $\TM$ whose morphisms are continuous functions $f:X\to Y$ between topologized magmas such that $f(1_X)=1_Y$ and $f(x\cdot y)=f(x)\cdot f(y)$ for all $x,y\in X$. Such functions $f$ will be called {\em continuous homomorphisms} of unital topologized magmas. 
It is clear that each (para)topological group is a unital topological magma.

A bijective self-map $\alpha:X\to X$ of a topologized magma $X$ is called an {\em automorphism} of $X$ if $\alpha$ and $\alpha^{-1}$ are continuous homomorphisms.  It is easy to see that the unit of a unital topologized magma is a fixed point of any automorphisms of $X$. The set $\Aut(X)$ of automorphisms of a topologized magma $X$ is a group with respect to the operation of composition. 

\begin{definition}\label{d:sep}
A unital topologized magma $X$ is called 
\begin{itemize}
\item {\em left separable} if $X$ contains a countable subset $S$ such that $SU=X$ for any neighborhood $U\subseteq X$ of the unit;
\item {\em right separable} if $X$ contains a countable subset $S$ such that $US=X$ for any neighborhood $U\subseteq X$ of the unit;
\item {\em Roelcke separable} if $X$ contains a countable subset $S$ such that $(US)U=X=U(SU)$ for any neighborhood $U\subset X$ of the unit;
\item {\em duoseparable} if $X$ contains a countable subset $S$ such that $S(US)=X=(SU)S$ for any neighborhood $U\subseteq X$ of the unit.
\end{itemize}
\end{definition}

The example of  commutative topological monoid $X=\{0\}\cup [1,\infty)\subset \IR$,  endowed with the operation of addition, shows that in topological monoids the separability is not equivalent to the duo (left, right, or Roelcke) separability.

There exists a simple contruction of non-separable duoseparable unital topologized magmas, which exploits the notions of  absorbing subgroup and absorbing automorphism.

A subgroup $H$ of the automorphism group $\Aut(X)$ of a unital topologized magma $X$ is called {\em absorbing} if $X=\bigcup_{h\in H}h(U)$ for every neighborhood $U\subseteq X$ of the unit. 

For a topologized magma $X$ and a subgroup $H\subseteq \Aut(X)$, endowed with the discrete topology, the {\em semidirect product} $X\rtimes H$ is the product $X\times H$ endowed with the binary operation $*$ defined by
$$(x,f)*(y,g)=(x\cdot f(y),f\circ g).$$

It is known (and easy to check) that the operation $*$ inherits many properties of the binary operation $\cdot$ of $X$: if $\cdot$ is continuous, associative, has a unit or is a group operation, then so is the operation $*$. If $1_X$ is the unit of the magma $X$, then the element $(1_X,1_H)$ is the unit of the magma $X\rtimes H$.

\begin{proposition}\label{p:autoH} Let $X$ be a unital topologized magma and $H$ be a countable subgroup of $\Aut(X)$. If $H$ is absorbing, then the semidirect product $X\rtimes H$ is a duoseparable unital topologized magma.
\end{proposition}

\begin{proof} Identify $X$ and $H$ with the submagmas $X\times\{1_H\}$ and $\{1_X\}\times H$ in $X\rtimes H$. To see that the unital topologized magma  $Y:=X\rtimes H$ is duoseparable, it suffices to show that $Y=(H*U)*H=H*(U*H)$ for any neighborhood $U\subseteq Y$ of the unit. Given any element $(x,h)\in Y$, use the absorbing property of the subgroup $H$ to find $g\in H$ such that $x=g(u)$ for some $u\in U\cap X$. Then
$$((1_X,g)*(u,1_H))*(1_X,g^{-1}h)=(g(u),g)*(1_X,g^{-1}h)=(x,g)*(1_X,g^{-1}h)=(x{\cdot}g(1_X),gg^{-1}h)=(x,h)$$
and 
$$(1_X,g)*((u,1_H)*(1_X,g^{-1}h))=(1_X,g)*(u\cdot 1_H(1_X),1_Hg^{-1}h)=(1_X,g)*(u,g^{-1}h)=(x,h).$$
Consequently, $Y=(H*U)*H=H*(U*H)$.
\end{proof}

An automorphism $\alpha$ of a unital topologized magma $X$ is called {\em absorbing} if the cyclic subgroup $\{\alpha^n\}_{n\in\IZ}$ generated by $\alpha$ in $\Aut(X)$ is absorbing. The latter means that $X=\bigcup_{n\in\IZ}\alpha^n(U)$ for any neighborhood $U$ of the unit in $X$.

Observe that for any topological vector space $X$ the automorphism $\alpha:X\to X$, $\alpha:x\mapsto x+x$, is absorbing.

For an automorphism $\alpha$ of a topologized magma $X$, the {\em semidirect product} $X\rtimes_\alpha\IZ$ is the product $X\times\IZ$ endowed with the binary operation $*$ defined by
$$(x,n)*(y,m)=(x\cdot\alpha^n(y),n+m).$$
Here $\IZ$ is the additive group of integer numbers, endowed with the discrete topology.

It is known (and easy to check) that the operation $*$ inherits many properties of the binary operation $\cdot$ of $X$: if $\cdot$ is continuous, associative, has a unit or is a group operation, then so is the operation $*$. If $1_X$ is the unit of $X$, then the element $(1_X,0)$ is the unit of the magma $X\rtimes_\alpha\IZ$. The following proposition can be proved by analogy with Proposition~\ref{p:autoH}.

\begin{proposition}\label{p:auto} If an automorphism $\alpha$ of a unital topologized magma $X$ is absorbing, then the semidirect product $X\rtimes_\alpha\IZ$ is a duoseparable unital topologized magma.
\end{proposition}

Next, we show that each unital topologized magma can be enlarged to a unital topologized magma admitting an absorbing automorphism.

The construction of such an extension exploits the famous Hartman-Mycielski construction \cite{BM}, \cite{HM}.
For any topological space $X$, the {\em Hartman-Mycielski space} $\HM(X)$ 
consists of all functions $f:[0,1)\to X$ for which there is a sequence $0=a_0<a_1<\dots<a_n=1$ such that for every $i<n$ the restriction $f{\restriction}[a_i,a_{i+1})$ is constant.

The space $\HM(X)$ carries the topology generated by the subbase consisting of the sets
$$N(U,a,b;\e)=\{f\in\HM(X):\lambda\big([a,b)\setminus f^{-1}(U)\big)<\e\}$$where $U$ is an open set in $X$, $0\le a<b\le 1$, $\e>0$, and $\lambda$ is the Lebesgue measure on $[0,1)$. The construction $\HM$ determines a functor in the category of topological spaces and their continuous maps. To each continuous map $f:X\to Y$ between topological spaces, the functor $\HM$ assigns the map $\HM f:\HM(X)\to\HM(Y)$, $\HM f:\gamma\mapsto f\circ \gamma$.

Any binary operation $\cdot:X\times X\to X$ induces the binary operation $*:\HM(X)\times\HM(X)\to\HM(X)$, $*:(f,g)\mapsto fg$, where $(fg)(t)=f(t)\cdot g(t)$. It can be shown that the (separate) continuity of the binary operation $\cdot$ implies the (separate) continuity of the binary operation $*$. Moreover, if the magma $(X,\cdot)$ has a unit $1_X$, then the constant map $1_{\HM(X)}:\HM(X)\to\{1_X\}\subset X$ is the unit of the magma $(\HM(X),*)$.

It is easy to see that for any continuous homomorphism $h:X\to Y$ between topologized magmas the continuous map $\HM h:\HM(X)\to\HM(Y)$ is a homomorphism between the magmas $\HM(X)$ and $\HM(Y)$. 

Therefore, the Hartman--Mycielski construction determines a functor $\HM$ in the category of (unital) topologized magmas. 

For every unital topologized magma $X$ with unit $1_X$, consider the subspace $$\HM_0(X)=\{f\in\HM(X):f(0)=1_X\}$$in the Hartman--Mycielski space $\HM(X)$. Observe that for any continuous homomorphism $f:X\to Y$ between unital topologized magmas, we have $\HM f(\HM_0(X))\subset\HM_0(Y)$, which allows us to consider $\HM_0$ as a functor $\HM_0:\TM\to\TM$ in the category $\TM$ of unital topologized magmas.

For any unital topologized magma $X$, consider the homomorphism $i_X:X\to\HM_0(X)$ assigning to each $x\in X$ the function $i_x:[0,1)\to\{x,1_X\}\subset X$ such that $i_x^{-1}(1_X)=[0,\frac12)$ and $i_x^{-1}(x)=[\frac12,1)$. It is easy to see that the map $i_X$ is  a topological embedding and the maps $i_X$ compose a natural transformation of the identity functor to the functor $\HM_0$. Therefore, we obtain the following useful fact.

\begin{proposition}\label{p:emb} Every unital topologized magma $X$ is topologically isomorphic to a submagma of the unital topologized magma $\HM_0(X)$.
\end{proposition}

Let $s:[0,1)\to[0,1)$ be the homeomorphism of the half-interval $[0,1)$ defined by $s(t)=t^2$. For every unital topologized magma $X$, the homeomorphism $s$ induces the automorphism $$\alpha_X:\HM_0(X)\to\HM_0(X),\;\;\alpha_X:\gamma\mapsto \gamma\circ s,$$ of the unital topologized magma $\HM_0(X)$.

The automorphism $\alpha_X$ is natural in the sense that for any continuous homomoprhism $h:X\to Y$ between unital topologized magmas, the following diagram commutes.
$$
\xymatrix{
\HM_0(X)\ar^{\HM_0 h}[r]\ar_{\alpha_X}[d]&\HM_0(Y)\ar^{\alpha_Y}[d]\\
\HM_0(X)\ar_{\HM_0 h}[r]&\HM_0(Y)\\
}
$$

\begin{proposition}\label{p:alpha} For any unital topologized magma $X$, the automorphism $\alpha_X$ of $\HM_0(X)$ is absorbing.
\end{proposition}

\begin{proof} Let $U$ be any neighborhood of the unit in $\HM_0(X)$. By the definition of the topology on $\HM(X)$, there exists a neighborhood $V_1$ of the unit $1_X$ in $X$ and $\e>0$ such that $\{f\in HM_0(X):\lambda(f^{-1}(V_1))>1-\e\}\subseteq U$. Given any $f\in\HM_0(X)$, we should find $n\in\IZ$ such that the function $\alpha_X^n(f)=f\circ s^n$ belongs to the neighborhood $U$. Since $f\in\HM_0(X)$, there exits a positive real number $b<1$ such that $f([0,b))=\{1_X\}\subseteq V_1$. Find $n\in\IN$ such that $(1-\e)^n<b$. Then for any $t\in[0,1-\e)$, we have $f\circ s^n(t)\in f([0,b))\subseteq V_1$ and hence the set $(f\circ s^n)^{-1}(V_1)\supseteq [0,1-\e)$ has Lebesgue measure $>1-\e$. Then $f\circ s^n\in U$.
\end{proof}

Propositions~\ref{p:auto}, \ref{p:emb} and \ref{p:alpha} imply our main result:

\begin{theorem}\label{t:main} For any unital topologized magma $X$, the semidirect product $\HM_0(X)\rtimes_{\alpha_X}\IZ$ is a duoseparable unital topologized magma, containing an isomorphic copy of $X$. 
\end{theorem}

If $X$ is a topological group, then so is $\HM_0(X)$ and $\HM_0(X)\rtimes_{\alpha_X}\IZ$. This observation, combined with Theorem~\ref{t:main},
implies the following corollaries.

\begin{corollary} For any topological group $X$, the topological group $\HM_0(X)\rtimes_{\alpha_X}\IZ$ is duoseparable and contains an isomorphic topological copy of $X$.
\end{corollary}

\begin{corollary}\label{c:main} Every topological group is topologically isomorphic to a subgroup of a duoseparable topological group.
\end{corollary}

It is easy to see that an abelian commutative topological group is separable if and only if it is duoseparable, so we cannot embed any abelian topological group into a duoseparable abelian topological group. 
\smallskip

A topological group $G$ is called {\em abelian-by-countable} (resp. {\em abelian-by-cyclic}) if $G$ contains a closed normal subgroup $H$ such that the quotient group $G/H$ is countable (resp. cyclic). Observing that for any abelian topological group $X$, the topological group $\HM_0(X)$ is abelian and $\HM_0(X)\rtimes_{\alpha_X}\IZ$ is abelian-by-cyclic, we obtain the following embedding result.

\begin{theorem} Every abelian topological group is  topologically isomorphic to a subgroup of a duoseparable abelian-by-cyclic topological group.
\end{theorem}

\begin{problem}\label{prob:emb} Can any each compact topological group be embedded into a duoseparable locally compact topological group?
\end{problem}

Let us observe that for compact topological groups the duoseparability is equivalent to the separability.

\begin{proposition} A compact topological group is duoseparable if and only if it is separable.
\end{proposition}

\begin{proof} The ``if'' part is trivial. To prove the ``only if'' part, assume that $G$ is duoseparable and find a countable subgroup $S\subseteq G$ such that $G=SUS$ for any neighborhood $U\subseteq G$ of the unit. We claim that the subgroup $S$ is dense in $G$ and hence $G$ is separable. In the opposite case, the closure $\bar S$ of $S$ is a proper closed subgroup of $G$. Fix any point $x\in G\setminus\bar S$ and using the compactness of the set $\bar S=\bar S\cdot 1_G\cdot\bar S$, find a neighborhood $U\subseteq G$ of $1_G$ such that $\bar S\cdot U\cdot\bar S\subseteq G\setminus\{x\}$. Then $SUS\subseteq \bar SU\bar S$ cannot be equal to $G$, which is a desired contradiction.
\end{proof}

\begin{proposition} Each duoseparable locally compact topological group $G$ is $\sigma$-compact.
\end{proposition}

\begin{proof} The duoseparability of $G$ yields a countable subset $S\subseteq G$ such that $G=SUS$ for any neighborhood $U$ of the unit in $G$. By the local compactness, the unit $1_G$ has a neighborhood $U$ with compact closure $\overline{U}$ in $G$. Then the space $G=S\overline US=\bigcup_{x,y\in S}x\overline Uy$ is $\sigma$-compact.
\end{proof}

The following theorem resolves the ``abelian'' case of Problem~\ref{prob:emb}.

\begin{theorem}\label{t:main2} Each $\sigma$-compact locally compact abelian topological group  is topologically isomorphic to a subgroup of a duoseparable abelian-by-countable locally compact topological group.
\end{theorem}

\begin{proof} We divide the proof of Theorem~\ref{t:main2} into a series of lemmas.

\begin{lemma}\label{l:SL} For any $n\in\IN$, $\e>0$, and vectors $x_1,\dots,x_n\in\IR^{2n}$, there exists a matrix $A\in SL(2n,\IZ)$ such that $x_iA\in [-\e,\e]^n\times\IR^n$ for every $i\in\{1,\dots,n\}$.
\end{lemma}

\begin{proof}  Let $m=2n$ and for every $i\in\{1,\dots,n\}$ write the vector $x_i\in\IR^{2n}=\IR^m$ in coordinates as $x_i=(x_{i1},\dots,x_{im})$. Consider the $n\times m$ matrix $X=(x_{ij})$. A column
of a matrix will be called {\em small} if all its entries have absolute value at most $\varepsilon$. Let $k$ be the maximal number of small columns in a matrix $XA$, where $A\in SL(m,\mathbb Z)$. Fix a matrix $A\in SL(m,\mathbb Z)$ such that $XA$ has $k$ many small columns.
Composing $A$ with a suitable coordinate permuting matrix, we can additionally assume that the small columns of the matrix $XA$ coincide with its first $k$ columns.

We claim that $k\ge n$. To derive a contradiction, assume that $k<n$ and hence $l:=m-k>n$. Let $y_1,\dots,y_l$ be the last $l$ columns of the matrix $XA$. After transposition, each column $y_i$ becames a row $y_i^T$ that belongs to the space $\IR^n$.

 Pick a positive integer $M$ such that the maximal absolute value of any element of the matrix $XA$ does not exceed $M\varepsilon$. Pick an arbitrary positive integer $K>(2lM)^n$ and consider the map $$f:\{-K,\dots,K\}^l\to \IR^n,\;f:(d_1,\dots,d_n)\mapsto \sum_{i=1}^ld_iy_i^T.$$  Since each $l$-tuple $d=(d_1,\dots,d_l)\in\{-K,\dots,K\}^l$ has  $\max_{i}|d_i|\le K$ and all entries of columns $y_i$ have absolute value at most $M\varepsilon$, each coordinate of the vector $f(d)$ is at most $lKM\varepsilon$. Therefore the image $f(\{-K,\dots,K\}^l)$ can be covered by $(2lKM)^n$ closed axis-parallel cubes with side $\varepsilon$. Since $|\{-K,\dots,K\}^l|=(2K+1)^l>(2K)^{n+1}>(2lKM)^n$, there exist two distinct elements $d'$ and $d''$ of  $\{-K,\dots,K\}^l$ such that each coordinate of the vector $y=f(d')-f(d'')=f(d'-d'')$ has absolute value at most $\varepsilon$. Put $d=d'-d''$. Dividing the elements of the $l$-tuple $d$ by their greatest common divisor, if needed, we can assume that the greatest common divisor of the entries of $d$ is $1$. By \cite{She}\footnote{See also {\tt https://kconrad.math.uconn.edu/blurbs/ringtheory/primvector.pdf}}, there exists a matrix $D\in SL(l,\mathbb Z)$ whose last column coincides with $d^T$. Put $B=\begin{pmatrix} I & 0\\ 0 & D\end{pmatrix}$, where $I$ is the $k\times k$ identity matrix. Then the last column of the matrix $XAB$ is small, whereas its first $k$ columns are the same as in the matrix $XA$ and hence are small, too. Consequently, the matrix $XAB$ has more than $k$ small columns, which  contradicts the definition of the number $k$. This contradiction shows that $k\ge n$ and hence $x_iA\subset[-\e,\e]^k\times\IR^{n-k}\subseteq[-\e,\e]^n\times\IR^n$ for all $i\in\{1,\dots,n\}$.
\end{proof}

Let $\IT=\IR/\IZ$ denote the torus group.

\begin{lemma}\label{l:Talpha} For any finite set $F\subset \IT^\w$ and any neighborhood $U\subseteq \IT^\w$ of zero there exists an automorphism $\alpha$ of $\IT^\w$ such that $\alpha(F)\subset U$.
\end{lemma}

\begin{proof} Consider the quotient homomorphism $q:\IR^\w\to\IR^\w/\IZ^\w=\IT^\w$ and observe that $q^{-1}(U)$ is a neighborhood of zero in $\IR^\w$. Choose a finite set $F'\subset\IR^\w$ of cardinality $|F'|=|F|$ such that $q(F')=F$.

By definition of the Tychonoff product topology on $\IR^\w$, there exists a finite set $D\subset\w$ and $\e>0$ such that $\pr_D^{-1}([-\e,\e]^D)\subset q^{-1}(U)$, where $\pr_D:\IR^\w\to\IR^D$, $\pr_D:x\mapsto x{\restriction}D$, stands for the natural projection. Replacing $D$ by a larger set, we can assume that $|D|\ge |F'|$.
By Lemma~\ref{l:SL},  there exists an automorphism $A$ of $\IR^\w$ such that $A(\IZ^\w)=\IZ^\w$ and $A(F')\subset \pr_D^{-1}([-\e,\e]^D)\subset q^{-1}(U)$.  Since $A(\IZ^\w)=\IZ^\w$, the authomorphism $A$ induces a unique automorphism $\alpha$ of $\IT^\w$ such that $\alpha\circ q=q\circ A$. Then $$\alpha(F)=\alpha\circ q(F')=q\circ A(F')\subset q(q^{-1}(U))=U.$$
\end{proof}

\begin{lemma}\label{l:Aw} There exists a countable subset $A\subset \Aut(\IT^\w)$ such that 
for any finite set $F\subset \IT^\w$ and any neighborhood $U\subseteq \IT^\w$ of zero there exists an automorphism $\alpha\in A$ of $\IT^\w$ such that $F\subset \alpha(U)$.
\end{lemma}

\begin{proof} Fix a countable neighborhood base $\{U_n\}_{n\in\w}$ of zero in the compact topological group $\IT^\w$ such that $U_{n+1}\subseteq U_n$ for all $n\in\w$. For every $n\in\w$ and every $x\in (\IT^\w)^n$, apply Lemma~\ref{l:Talpha} and find an automorphism $\alpha_x$ of $\IT^\w$ such that $\{x(i)\}_{i\in n}\subset \alpha_x(U_n)$. Observe that for every $x\in (\IT^\w)^n$ the set $W_n=\{y\in (\IT^\w)^n:\{y(i)\}_{i\in n}\subset \alpha_x(U_n)\}$ is an open neighborhood of $x$. By the compactness of $(\IT^\w)^n$, the open cover $\{W_x:x\in(\IT^\w)^n\}$ has a finite subcover. Consequently, there exists a finite set $E_n\subset (\IT^\w)^n$ such that $(\IT^\w)^n=\bigcup_{x\in E_n}W_n$. Consider the finite set of automorphisms $A_n=\{\alpha_x:x\in E_n\}$ and observe that for any set $F\subset\IT^\w$ of cardinality $|F|\le n$ there exists an automorphism $\alpha\in A_n$ such that $F\subset \alpha(U_n)$. Then the countable set $A=\bigcup_{n=1}^\infty A_n\subset\Aut(\IT^\w)$ has the required property.
\end{proof}

\begin{lemma}\label{l:Ak} For any infinite cardinal $\kappa$, the automorphism group $\Aut(\IT^\kappa)$ of the compact topological group $\IT^\kappa$ contains a countable absorbing subgroup.
\end{lemma}

\begin{proof} Since $\kappa$ is infinite, we can identify $\IT^\kappa$ with $(\IT^\w)^\kappa$. 
By Lemma~\ref{l:Aw}, there exists a sequence of automorphisms $\{\alpha_n\}_{n\in\w}\subset \Aut(\IT^\w)$ such that for every finite set $F\subset\IT^\w$ and any neighborhood $U\subseteq\IT^\w$ of zero,  there exists $n\in\w$ such that $F\subset \alpha_n(U)$. Each automorphism $\alpha_n$ induces the automorphism $$\tilde\alpha_n:(\IT^\w)^\kappa\to(\IT^\w)^\kappa,\;\;\tilde\alpha_n:(x_i)_{i\in\kappa}\mapsto (\alpha_n(x_i))_{i\in\kappa},$$
of the compact topological group $(\IT^\w)^\kappa$. We claim that the countable set  $A=\{\tilde \alpha_n\}_{n\in\w}\subset\Aut\big((\IT^\w)^\kappa\big)$ generates an absorbing subgroup $\langle A\rangle$ of $\Aut((\IT^\w)^\kappa)$. 

Fix any neighborhood of zero $U\subseteq(\IT^\w)^\kappa$ and any point $x=(x_i)_{i\in\kappa}\in (\IT^\w)^\kappa$. By the definition of the Tychonoff product topology, there exists a finite set $F\subset\kappa$ and a neighborhood $V\subseteq\IT^\w$ of zero such that $U\supseteq\pr_F^{-1}(V^F)$ where $\pr_F:(\IT^\w)^\kappa\to(\IT^\w)^F$, $\pr_F:x\mapsto x{\restriction}F$, is the natural projection. The choice of the sequence $\{\alpha_n\}_{n\in\w}$ guarantees that for the finite set $\{x(i)\}_{i\in F}\subset\IT^\w$ there exists $n\in\w$ such that $\{x(i)\}_{i\in F}\subset\alpha_n(V)$. Then $\pr_F(x)\in (\alpha_n(V))^F$ and $x\in \tilde\alpha_n(\pr^{-1}_F(V^F))\subseteq \tilde\alpha_n(U)$. Therefore, $(\IT^\w)^\kappa=\bigcup_{n\in\w}\tilde \alpha_n(U)=\bigcup_{\alpha\in\langle A\rangle}\alpha(U)$, which means that the subgroup $\langle A\rangle$ is absorbing.
\end{proof} 

\begin{lemma}\label{l:final1} For any infinite cardinal $\kappa$, the compact topological group $\IT^\kappa$ embeds into a duoseparable locally compact abelian-by-countable topological group.
\end{lemma}

\begin{proof} By Lemma~\ref{l:Ak}, the automorphism group $\Aut(\IT^\kappa)$ contains a countable absorbing subgroup $H\subseteq\Aut(\IT^\kappa)$. Endow $H$ with the discrete topology and observe that the semidirect product $\IT^\kappa\rtimes H$ is a locally compact abelian-by-countable topological group. By Proposition~\ref{p:auto}, the topological group $\IT^\kappa\rtimes H$ is duoseparable.
\end{proof}

\begin{lemma}\label{l:final2} Each $\sigma$-compact locally compact abelian group $G$ embeds into the product $\IR^n\times\IT^\kappa\times H$ for suitable $n\in\w$, cardinal $\kappa$ and countable discrete group $H$.
\end{lemma}

\begin{proof} By  the Principal Structure Theorem 25 for LCA-groups in \cite{Morris}, $G$ contains an open subgroup $G_0$, which is isomorphic to $\IR^n\times K$ for some $n\in\w$ and some abelian compact topological group $K$. By Theorem 14 in \cite{Morris}, the compact abelian topological group $K$ admits a continuous injective homomorphism $K\to\IT^\kappa$ for a suitable cardinal $\kappa$. Then the locally compact group $G_0\cong \IR^n\times K$ admits a homomorphism $h:G_0\to\IR^n\times\IT^\kappa$, which is a topological embedding.
Since the group $\IR^n\times\IT^\kappa$ is divisible, the homomorphism $h$ extends to a homomorphism $\bar h:G\to \IR^n\times\IT^\kappa$ (see Proposition 17 in \cite{Morris}). The homomorphism $\bar h$ is continuous because it has continuous restriction $\bar h{\restriction}G_0=h$ to the open subgroup $G_0$ of $G$. The $\sigma$-compactness of the topological group $G$ implies the $\sigma$-compactness of the quotient group $G/G_0$. Being discrete and $\sigma$-compact, the topological group $H=G/G_0$ is countable. Let $q:G\to G_0$ be the quotient homomorphism. Then $(\bar h,q):G\to (\IR^n\times\IT^\kappa)\times H$, $(\bar h,q):x\mapsto (\bar h(x),q(x))$, is the required isomorphic embedding of $G$ into the product $\IR^n\times \IT^\kappa\times H$. 
\end{proof}

Observing that the product of duoseparable (abelian-by-countable) topological groups is duoseparable (and abelian-by-countable), we see that Theorem~\ref{t:main2} follows from Lemmas~\ref{l:final1} and \ref{l:final2}.
\end{proof}
\smallskip

Corollary~\ref{c:main} has an interesting application to constructing duo narrow topological groups which are not narrow. The narrowness is a ``countable'' modification of the precompactness, whose various versions are defined as follows.

\begin{definition}\label{d:precomp} A unital topologized magma $X$ is called
\begin{itemize} 
\item {\em precompact} if for any neighborhood $U\subseteq X$ of the unit there exists a finite subset $S\subseteq X$ such that $SU=X=US$;
\item {\em left precompact} if for any neighborhood $U\subseteq X$ of the unit there exists a finite subset $S\subseteq X$ such that $SU=X$;
\item {\em right precompact} if for any neighborhood $U\subseteq X$ of the unit there exists a finite subset $S\subseteq X$ such that $US=X$;
\item {\em duo precompact} if for any neighborhood $U\subseteq X$ of the unit there exists a finite subset $S\subseteq X$ such that $(SU)S=X=S(US)$;
\item {\em Roelcke precompact} if for any neighborhood $U\subseteq X$ of the unit there exists a finite subset $S\subseteq X$ such that $(US)U=X=U(SU)$.
\end{itemize}
\end{definition}

For topological groups we have the following equivalence (see  \cite[15.81]{Us} or \cite[4.3]{BT}).

\begin{theorem}\label{t:precomp} For a topological group $X$ the following conditions are equivalent:
\begin{enumerate}
\item $X$ is precompact;
\item $X$ is left precompact;
\item $X$ is right precompact;
\item $X$ is duo precompact;
\item $X$ embeds into a compact topological group.
\end{enumerate}
\end{theorem}

The Roelcke precompactness is not equivalent to the precompactness because of the following fundamental result of Uspenski \cite{Us} (which nicely complements our Corollary~\ref{c:main}).

\begin{theorem}[Uspenskij] Each topological group embeds into a Roelcke precompact topological group.
\end{theorem}

Replacing the adjective ``finite'' in Definition~\ref{d:precomp} by ``countable'', we obtain various modifications of the narrowness.

\begin{definition}\label{d:n} A unital topologized magma $X$ is called
\begin{itemize} 
\item {\em narrow} if for any neighborhood $U\subseteq X$ of the unit there exists a countable subset $S\subseteq X$ such that $SU=X=US$;
\item {\em left narrow} if for any neighborhood $U\subseteq X$ of the unit there exists a countable subset $S\subseteq X$ such that $SU=X$;
\item {\em right narrow} if for any neighborhood $U\subseteq X$ of the unit there exists a countable subset $S\subseteq X$ such that $US=X$;
\item {\em duo narrow} if for any neighborhood $U\subseteq X$ of the unit there exists a countable subset $S\subseteq X$ such that $(SU)S=X=S(US)$;
\item {\em Roelcke narrow} if for any neighborhood $U\subseteq X$ of the unit there exists a countable subset $S\subseteq X$ such that $(US)U=X=U(SU)$.
\end{itemize}
\end{definition}

Narrow topological groups were introduced by Guran \cite{Gur} who proved that a topological group is narrow if and only if it embeds into a Tychonoff product of second-countable topological groups. In \cite{Gur} and \cite{AT,Tka} narrow topological groups are called $\w$-bounded and $\w$-narrow, respectively. By Lemma 3.31 in \cite{Pachl}, a topological group $X$ is narrow if and only if for any neighborhood $U\subseteq X$ of the identity there exist a countable set $C\subseteq X$ and a finite set $F\subseteq X$ such that $CUF=X$.

For topological groups the properties introduced in the Definitions~\ref{d:sep}, \ref{d:precomp}, \ref{d:n} relate as follows.
{\small
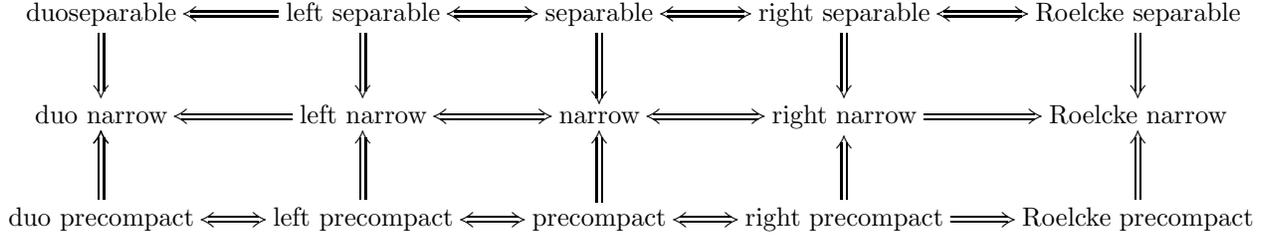
$$\xymatrix{
\mbox{duoseparable}\ar@{=>}[d]&\mbox{left separable}\ar@{=>}[l]\ar@{<=>}[r]\ar@{=>}[d]&\mbox{separable}\ar@{<=>}[r]\ar@{=>}[d]&\mbox{right separable}\ar@{<=>}[r]\ar@{=>}[d]&\mbox{Roelcke separable}\ar@{=>}[d]\\
\mbox{duo narrow}&\mbox{left narrow}\ar@{<=>}[r]\ar@{=>}[l]&\mbox{narrow}\ar@{<=>}[r]&\mbox{right narrow}\ar@{=>}[r]&\mbox{Roelcke narrow}\\
\mbox{duo precompact}\ar@{=>}[u]\ar@{<=>}[r]\ar@{=>}[u]&\mbox{left precompact}\ar@{<=>}[r]\ar@{=>}[u]&\mbox{precompact}\ar@{<=>}[r]\ar@{=>}[u]&\mbox{right precompact}\ar@{=>}[r]\ar@{=>}[u]&\mbox{Roelcke precompact}\ar@{=>}[u]
}
$$
}

Looking at Theorem~\ref{t:precomp}, it is natural to ask if the narrowness is equivalent to the duo narrowness. But this is not true.

\begin{example} Take any nonseparable Banach space $H$, and consider the absorbing automorphism $\alpha:H\to H$, $\alpha:x\mapsto x+x$, and the semidirect product $X=H\rtimes_\alpha \IZ$. By Proposition~\ref{p:auto}, the topological group $X$ is duoseparable and hence duo narrow but it is not narrow (as the narrowness is inherited by subgroups).
\end{example} 

There exists an interesting property generalizing both the precompactness and the separability.

\begin{definition} A unital topologized magma $X$ is called {\em preseparable} if there exists a countable subset $S\subseteq X$ such that for any neighborhood $U\subseteq X$ of the unit there exists a finite set $F\subseteq X$ such that $$X=S(UF)=(SU)F=F(US)=(FUS).$$
\end{definition}

It is is clear that a topological group $X$ is preseparable if it is separable or precompact. More generally, $X$ is preseparable if it contains a separable closed normal subgroup $H\subseteq X$ whose quotient group $X/H$ is precompact.

By Lemma 3.31 in \cite{Pachl}, each preseparable abelian topological group is narrow. On the other hand, the Tychonoff power $\IR^{\mathfrak c^+}$ of $\mathfrak c^+$ many copies of the real line is an example of an abelian topological group which is narrow but is not preseparable. 

Therefore, we have the following implications holding for every topological group.
$$
\xymatrix{
\mbox{precompact}\ar@{=>}[r]&\mbox{preseparable}\ar@{=>}[d]&\mbox{separable}\ar@{=>}[l]\\
&\mbox{narrow}
}
$$

\section*{Acknowledgement}

The authors express their sincere thanks to Jan Pachl for his valuable comment on  the characterization of narrow topological groups, given in Lemma 3.31 in his book \cite{Pachl}.


\newpage


\begin{thebibliography}{}

\bibitem{AT} A.~Arhangel'skii, M.~Tkachenko, {\em Topological groups and related structures}, Atlantis Studies in Mathematics, 1. Atlantis Press, Paris; World Scientific Publishing Co. Pte. Ltd., Hackensack, NJ, 2008.

\bibitem{BGR}  T.~Banakh, I.~Guran, A.~Ravsky, {\em Generalizing separability, precompactness and narrowness in topological groups}, preprint ({\tt arxiv.org/abs/2002.06459}).

\bibitem{BT} A.~Bouziad, J.-P.~Troallic, {\em A precompactness test for topological groups in the manner of Grothendieck}, Topology Proc. {\bf 31}:1 (2007), 19--30.

\bibitem{BM} R.~Brown, S.~Morris, {\em Embeddings in contractible or compact objects}, Colloq. Math. {\bf 38}:2 (1977/78), 213--222.

\bibitem{Gur} I.I.~Guran, {\em Topological groups similar to Lindel\"of groups}, Dokl. Akad. Nauk SSSR {\bf 256}:6 (1981), 1305--1307.

\bibitem{HM} S.~Hartman, J.~Mycielski, {\em On the imbedding of topological groups into connected topological groups}, Colloq. Math. {\bf 5} (1958) 167--169.

\bibitem{Morris} S.A.~Morris, {\em Pontryagin duality and the structure of locally compact abelian groups}, London Mathematical Society Lecture Note Series, No. 29. Cambridge University Press, Cambridge-New York-Melbourne, 1977. 

\bibitem{Pachl} J.~Pachl, {\em Uniform spaces and measures}, Fields Institute Monographs, {\bf 30}, Springer, 2013.

\bibitem{She} E.~Schenkman, {\em The basis theorem for finitely generated Abelian groups}, Amer. Math. Monthly {\bf 67} (1960), 770--771.



\bibitem{Tka} M.~Tka\v cenko, {\em Introduction to topological groups}, Topology Appl. {\bf 86}:3 (1998), 179--231.

\bibitem{Us} V.V.~Uspenskij, {\em On subgroups of minimal topological groups}, Topology Appl. {\bf 155}:14 (2008), 1580--1606.

\end{thebibliography}
\end{document}